\documentclass[a4paper,12pt]{amsart}

\usepackage[centertags]{amsmath}
\usepackage{amsfonts,amssymb,amstext,amsthm,newlfont,latexsym,stmaryrd}
\usepackage[left=3.0cm,right=3.0cm,top=3.7cm,bottom=3.7cm]{geometry}
\usepackage{tikz,hyperref}
\usetikzlibrary{arrows,automata,positioning,fit,shapes}

\theoremstyle{plain}
\newtheorem{theorem}{Theorem}[section]
\newtheorem{corollary}[theorem]{Corollary}
\newtheorem{lemma}[theorem]{Lemma}
\newtheorem{proposition}[theorem]{Proposition}
\theoremstyle{definition}
\newtheorem{definition}[theorem]{Definition}
\newtheorem{example}[theorem]{Example}
\theoremstyle{remark}
\newtheorem{remark}[theorem]{Remark}

\makeatletter
  \DeclareRobustCommand{\cev}[1]{%
    \mathpalette\do@cev{#1}%
  }
  \newcommand{\do@cev}[2]{%
    \fix@cev{#1}{+}%
    \reflectbox{$\m@th#1\vec{\reflectbox{$\fix@cev{#1}{-}\m@th#1#2\fix@cev{#1}{+}$}}$}%
    \fix@cev{#1}{-}%
  }
  \newcommand{\fix@cev}[2]{%
    \ifx#1\displaystyle
      \mkern#23mu
    \else
      \ifx#1\textstyle
        \mkern#23mu
      \else
        \ifx#1\scriptstyle
          \mkern#22mu
        \else
          \mkern#22mu
        \fi
      \fi
    \fi
  }
\makeatother

\begin{document}

\title[Loss of Gibbs property]{Loss of Gibbs property in one-dimensional mixing shifts of finite type}

\author[S. Hong]{Soonjo Hong} \address{Hongik University \\
    2639, Sejong-ro, Jochiwon-eup 
    \\
    Sejong\\
    South Korea}
\email{hsoonjo@hongik.ac.kr}

\date{}
\begin{abstract}
  Let $\pi$ be a factor map from a one-dimensional mixing shift of finite type $X$ onto a sofic shift $Y$.
  We investigate when $\pi$ sends Gibbs measures on $X$ to non-Gibbs measures on $Y$.
\end{abstract}
\maketitle

\section{Introduction}

The preservation and loss of Gibbs property of random fields under renormalisation transformation have been discovered and studied in statistical mechanics \cite{JozCK98,vanFS93}.
From the viewpoint of symbolic dynamics, Gibbs states can be formulated as Gibbs measures on shift spaces \cite{Bow08,Sin72}.
In this context, renormalisation transformations are regarded as factor maps between shift spaces.
  
In the present paper, we study which factor maps lose Gibbs property of measures when the domains are one-dimensional mixing shifts of finite type.
Such a problem was studied in \cite{ChaU03,Yoo10} using linear algebra, in \cite{Kem11} constructing concrete potential functions and computing their variation, in \cite{Ver11} disintegrating a measure into non-homogeneous equilibrium states and in \cite{Pir18} applying cone techniques and operator theory.
All the results found in \cite{ChaU03,Kem11,Ver11,Yoo10} suggest a crucial property for factor maps to preserve Gibbs property, named by Yoo in \cite{Yoo10} as {\em fiber-mixing} property.
We investigate other cases where factor maps are not fiber-mixing.
  
The notion of transition class is applied to the study.
Transition class was devised in \cite{AllQ13} and further developed in \cite{AllHJ15,AllHJ15E} to study factor maps from shifts of finite type.
The author will narrow down possible candidates for factor maps preserving Gibbs property excluding several cases of factor maps under some conditions on transition classes.


  
\section{Backgrounds}

We assume basic knowledge of symbolic dynamics here.
If one needs more exposition about symbolic dynamics, refer to \cite{LinM95}.
In this section, we introduce transition classes, fiber-mixing factor maps and Gibbs measures.

\subsection{Transition class}
From now on, let $\mathcal{A}$ be a finite alphabet.
Recall that a shift of finite type $X$ is 1-step if and only if for any $uv$ and $vw$ in $\mathcal{B}(X)$ with $|v|=1$ we have $uvw$ in $\mathcal{B}(X)$.
Any shift of finite type is conjugate to a 1-step shift of finite type.
In this paper, let $\pi:X\to Y$ be a 1-block factor map from a one-dimensional two-sided 1-step shift of finite type $X$ over $\mathcal{A}$ onto a sofic shift $Y$ over $\mathcal{A}$, unless stated otherwise.

\begin{definition}\label{defn::bridge}
  Let $u$ and $w$ be in the set $\mathcal{B}_l(X)$ of words of length $l$ in $X$ for some $l>0$ with $\pi(u)=\pi(w)$.
  Then a path $v$ in $\mathcal{B}_l(X)$ is called a {\em bridge} from $u$ to $w$ if
  \[ v|_1=u|_1,v|_l=w|_l\text{ and }\pi(v)=\pi(u)=\pi(w). \]
  A pair of bridges from $u$ to $w$ and from $w$ to $u$ is called a {\em two-way bridge} between $u$ and $w$.
\end{definition}

\begin{definition}\label{defn::class-degree}
  Given $m$ in $\mathbb Z$, $y$ in $Y$ and $x,x'$ in $\pi^{-1}(y)$ a {\em right} $m$-{\em bridge} from $x$ to $x'$ is another preimage $\vec x$ of $y$ such that for some $n>m$ we have
  \[ x|_{(-\infty,m]}=\vec x|_{(-\infty,m]}\text{ and }\vec x|_{[n,\infty)}=x'|_{[n,\infty)}. \]
  A {\em right transition} from $x$ to $x'$ is a sequence $\{\vec x^{(m)}\}_{m\in\mathbb{Z}}$ of $m$-bridges from $x$ to $x'$.
  When there is a right transition from $x$ to $x'$, we write $x\to^rx'$.

  We say that $x$ and $x'$ are {\em right equivalent} and write $x\sim^rx'$ if $x\to^rx'$ and $x'\to^rx$.
  It is indeed an equivalence relation and the equivalence class $[x]^r$ of $x$ up to $\sim^r$ is called a {\em right (transition) class}.
  Set $\llbracket y\rrbracket^r=\{[x]^r\mid x\in\pi^{-1}(y)\}$ and $d_\pi^r(y)=|\llbracket y\rrbracket^r|$ for $y$ in $Y$.
  The {\em right class degree} $d_\pi^r$ of $\pi$ is defined to be $d_\pi^r=\min_{y\in Y}d_\pi^r(y)$.
\end{definition}
    
\begin{definition}\label{defn::nonstop_transition}
  Let $y$ be in $Y$ with distinct right classes $C\to^rC'$.
  If for any right class $C''\ne C,C'$ over $y$ there is no transition $C\to^rC''$ nor $C''\to^rC'$,
  then the transition $C\to^rC'$ is said to be {\em nonstop}.
\end{definition}


We may consider a {\em left} $m$-bridge from $x$ to $x'$ which is a preimage $\vec x$ of $y$ such that for some $n<m$ we have
  \[ x'|_{(-\infty,n]}=\vec x|_{(-\infty,n]}\text{ and }\vec x|_{[m,\infty)}=x|_{[m,\infty)}. \]
Subsequently we consider a {\em left transition} $x\to^lx'$ from $x$ to $x'$, a sequence $\{\vec x^{(m)}\}_{m\in\mathbb Z}$ of left $m$-bridges from $x$ to $x'$.
Then {\em left equivalence} $x\sim^lx'$ and the {\em left class}  $[x]^l$ of $x$ up to $\sim^l$ are considered as well.
If for $y$ in $Y$ we put $\llbracket y\rrbracket^l=\{[x]^l\mid x\in\pi^{-1}(y)\}$ and $d_\pi^l(y)=|\llbracket y\rrbracket^l|$,
then the {\em left class degree} $d_\pi^l=\min_{y\in Y}d_\pi^l(y)$ of $\pi$ is equal to $d_\pi^r$ \cite[\S 6]{AllHJ15}.
Hence we may omit ``left'' or ``right'' in front of {\em class degree} and just denote it by $d_\pi$.
Still, a point in $Y$ may have distinct numbers of right and left classes.
So we will continue to distinguish $d_\pi^r(\cdot)$ from $d_\pi^l(\cdot)$.
  
\begin{definition}\label{defn::routable_blocks}
  A word $u$ in $\mathcal{B}(X):=\bigcup_{n\ge0}\mathcal{B}_n(X)$ is said to be {\em routable} through $M$ at $n$ for some $1<n<|u|$ and $M\subset\mathcal{A}$
  if there is $v$ in $\mathcal{B}(X)$ with
  \[ |u|=|v|,u|_1=v|_1,u|_{|u|}=v|_{|u|},\pi(u)=\pi(v)\text{ and }v|_n\in M. \]
  A word $w$ in $\mathcal{B}(Y)$ is said to be {\em fiber-routable} through $M$ at $n$ if all the words in $\pi^{-1}(w)$ are routable through $M$ at $n$. Define the {\em depth} $d_\pi(w)$ of $w$ by
  \[ d_\pi(w)=\min_{1<n<|w|}\min\{|M|\mid M\subset\mathcal{A},w\text{ is fiber-routable through }M\text{ at }n\}. \]
\end{definition}

  \begin{definition}\label{defn::tangled_blocks}
    Given $w$ in $\mathcal{B}(Y)$, a subset $\mathcal{W}$ of $\pi^{-1}(w)$ is said to be {\em tangled} if between any two words of $\mathcal{W}$ lies a two-way bridge.
    A partition of $\pi^{-1}(w)$ each member of which is tangled is said to be a {\em tangled} partition of $\pi^{-1}(w)$.
    The $\tau$-{\em depth} $\tau_\pi(w)$ of $w$ is the smallest cardinality of a tangled partition of $\pi^{-1}(w)$.
  \end{definition}
  
Every extension $uwv$ of a block $w$ in $\mathcal{B}(Y)$ has $\tau$-depth and depth no greater than $w$ has.
Also $\tau_\pi(w)\le d_\pi(w)$ and $d_\pi=\min_{w\in\mathcal{B}(Y)}d_\pi(w)=\min_{w\in\mathcal{B}(Y)}\tau_\pi(w)$ \cite{AllHJ15,AllHJ15E}.
  
A {\em recurrent point} is a point any word of which occurs infinitely often to the right of it.
  
\begin{theorem}\label{thm::depths_and_degree}\cite[Theorem 4.22]{AllHJ15E}
  For some $n$ in $\mathbb{Z}$ there are infinitely many occurrences of words with depth less than or equal to $d=d_\pi(y)$ in $y|_{[n,\infty)}$.
  If $y$ is recurrent, then
  \[ d=\min\{d_\pi(y|_{[m,n]})\mid m\le n\}=\min\{\tau_\pi(y|_{[m,n]})\mid m\le n\}. \]
\end{theorem}

\subsection{Fiber-mixing factor maps}

Fiber-mixing factor maps are known to preserve Gibbs property \cite{Kem11,Pir18,Ver11,Yoo10}.

\begin{definition}\label{defn::fiber-mixing}
  If given any $y$ in $Y$ and its preimage $x$ and $x'$ there is $\vec{x}$ which is left asymptotic to $x$ and right asymptotic to $x'$,
  then $\pi$ is said to be {\em fiber-mixing}.
\end{definition}

Fiber-mixing factor maps are characterised by other properties as well: {\em class-closing} properties and {\em continuing} properties.

\begin{definition}\label{defn::class-closing}
  If any left asymptotic preimages $x$ and $x'$ of $y$ are right equivalent,
  then $\pi$ is said to be {\em right class-closing}.
\end{definition}

\begin{definition}\label{defn::continuing}
  If given $x$ in $X$ and $y$ in $Y$ such that $\pi(x)$ is left asymptotic to $y$ there is some preimage $x'$ of $y$ left asymptotic to $x$,
  then $\pi$ is said to be {\em right continuing}.
\end{definition} 

In symmetric ways, we define {\em left} class-closing factor maps and {\em left} continuing factor maps.
It is also helpful that continuing factor maps are eresolving up to topological conjugacy \cite{BoyT84,Jun11,Yoo15}.

\begin{definition}\label{defn::eresolving}
  If given any $ab$ in $\mathcal{B}_2(Y)$ with $e$ in $\pi^{-1}(a)$ there is $f$ in $\pi^{-1}(b)$ with $ef$ in $\mathcal{B}_2(X)$,
  then $\pi$ is said to be {\em right eresolving}.
  {\em Left eresolving factor maps} are defined in a similar way and {\em bi-eresolving} factor maps are factor maps which are both left and right eresolving factor maps. 
\end{definition}

{\em Bi}-class-closing factor maps, {\em bi}-continuing factor maps and {\em bi}-eresolving factor maps refer to both left and right class-closing factor maps, continuing factor maps and eresolving factor maps, respectively.

\begin{theorem}\cite{AllHJ15}\label{thm::constant-class-to-one_and_class-closing_with_continuing}
  The following are equivalent:
  \begin{enumerate}
    \item $d_\pi^r(\cdot)$ is constant on $Y$.
    \item $d_\pi^l(\cdot)$ is constant on $Y$.
    \item $\pi$ is left class-closing and right continuing.
    \item $\pi$ is right class-closing and left continuing.
    \item $\pi$ is bi-class-closing and bi-continuing.
  \end{enumerate}
\end{theorem}

\begin{theorem}\cite{AllHJ15}\label{thm::fiber-mixing_conditions}
  The following are equivalent:
  \begin{enumerate}
    \item $\pi$ is fiber-mixing.
    \item For all $y$ in $Y$ $d_\pi^r(y)=d_\pi^l(y)=1$.
    \item $\pi$ is bi-class-closing and bi-continuing and $d_\pi=1$.
    \item There is $N>0$ such that for all $w$ in $\mathcal{B}_N(Y)$ $d_\pi(w)=1$.
    \item\label{condition::fiber-mixing_tangled} There is $N>0$ such that for all $w$ in $\mathcal{B}_N(Y)$ $\tau_\pi(w)=1$.
  \end{enumerate}
\end{theorem}

\subsection{Gibbs measures}

Let $X$ be a mixing shift of finite type over an alphabet $\mathcal{A}$.

\begin{definition}\label{defn::markov_measure}
  A $\sigma$-invariant measure $\mu$ on $X$ is called a {\em 1-step Markov} measure
  if there are a {\em initial} probability vector $p$ on $\mathcal{A}$ and an $\mathcal{A}\times\mathcal{A}$ stochastic matrix $P$ such that
  \[
     \mu[a_1a_2\cdots a_n]_0=p_{a_1}P_{a_1,a_2}P_{a_2,a_3}\cdots P_{a_{n-1},a_n}
  \]
  for all $a_1\cdots a_n$ in $\mathcal{B}(X)$.
  A measure on $X$ is called a {\em Markov} measure if it is conjugate to a 1-step Markov measure on a shift of finite type.
\end{definition}

Denote by $S_0^{n-1}f(x)=\sum_{j=0}^{n-1}f(\sigma^jx)$.
  
\begin{definition}\label{defn::gibbs_measure}
  Let $f$ be in $C(X)$.
  A $\sigma$-invariant measure $\mu$ on $X$ is called a {\em Gibbs} measure for $f$ if there are $K_1,K_2>0$ and a real number $P$ such that
  \begin{equation}\label{eqn::gibbs_inequality}
     K_1<\frac{\mu[x|_{[0,n-1]}]}{\exp(-nP+S_0^{n-1}f(x))}<K_2 
  \end{equation}
  for all $x$ in $X$ and $n$ in $\mathbb{N}$,
  where $f$ is called a {\em potential} of $\mu$ with {\em Gibbs constants} $K_1,K_2$.
  If $P=0$ in particular, then $f$ is called a {\em normalized potential} of $\mu$.
\end{definition}

Markov measures are Gibbs measures for locally constant potentials.
Also, we can easily verify that it is invariant under conjugacy whether the given measure is Gibbs or not.

If $\mu$ is a Gibbs measure for a potential $f$, then subtracting $P$ from $f$ we can always assume a normalized potential.
For the convenience, all the potentials in this paper are assumed normalized.

\section{Factor maps which admit transitions between right classes}

We start by showing that factor maps does not preserve Gibbsian property if it admits a right transition between distinct right classes over a periodic point.
Throughout the section $\pi$ is always a 1-block factor map from a two-sided 1-step mixing shift of finite type $X$ over $\mathcal{A}$ onto a sofic shift $Y$ over $\mathcal{A}$.
The following example exhibits a typical behaviour of such ones.

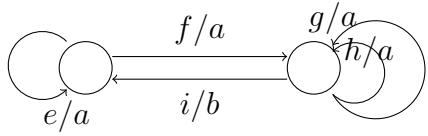
\begin{figure}
  \begin{tikzpicture}[->]
    \node[circle,draw,inner sep=7pt] (I) at (0,1) {};
	\node[circle,draw,inner sep=7pt] (J) at (3,1) {};
	      
	\draw (0.35,1.15) to (1.5,1.15) node[above] {$f/a$} to (2.65,1.15) ;
	\draw (2.65,0.85) to (1.5,0.85) node[below] {$i/b$} to (0.35,0.85) ;
	\draw (3.25,0.65) arc (225:495:0.4) node[right] {$h/a$};
	\draw (3.25,0.65) arc (210:510:0.65) node[above] {$g/a$};
	\draw (-0.25,1.35) arc (45:315:0.45) node[below] {$e/a$};
  \end{tikzpicture}
  \caption{There is a transition between two classes}\label{fig::transitional_cases}
\end{figure}
        
\begin{example}\label{eg::transitional_cases}
  The labelling $\pi$ given in Figure \ref{fig::transitional_cases} admits a transition between two right classes over a fixed point $a^\infty$.
  Define a Markov measure $\mu$ on $X$ putting transition probabilities $1/2$ on $e,f$ and $i$ and $1/4$ on each of $g$ and $h$.
  Set $\nu=\pi\mu$.
  Then
  \[ \nu[a^n]_0=\mu[e^n]_0+\sum_{m=0}^{n-1}\mu[e^mf(g\wedge f)^{n-m-1}]_0+\mu[(g\wedge f)^n]_0=(2+n)2^{-n}. \]
  On the other hand, for $f:X\to\mathbb{R}$ we have $\exp S_nf(a^\infty)=(\exp f(a^\infty))^n$.
  Since $n2^{-n}/(\exp f(a^\infty))^n$ goes either to 0 or $\infty$ for any $f$, $\nu$ never is a Gibbs measure.
\end{example}

As in the previous example, we will find points of $Y$
where the transformed measures violate the inequality (\ref{eqn::gibbs_inequality}) for any positive real constant and continuous function.
Such points are found among periodic points
as they behave more regularly and are easier to compute and control the complexity of their transition classes.
The desired periodic points and measures on their fibers are given by Lemma \ref{lmm::complexity_of_periodic_class_realized}.

First, we compute the approximate growth rate of the complexity of the transition classes of periodic points.
  
\begin{definition}\label{defn::marked_and_transient_blocks}
  Let $y$ be in $Y$, $I$ an interval in $\mathbb{Z}$, $w$ in $\pi^{-1}(y)|_I$ and $C$ a right class over $y$.
  We say $w$ to be {\em marked} in $I$ by $C$ and let $w$ belong to the set $\mathsf{mk}_I(C)$ if
  \[ \{C'\in\llbracket y\rrbracket^r\mid w\in C'|_I\}=\{C'\in\llbracket y\rrbracket^r\mid C\to C'\}. \]
  The blocks in $\pi^{-1}(y)|_I$ but not in $\bigcup_{C}\mathsf{mk}_I(C)$ are said to be {\em transient} through $I$.
  For the convenience, let $\mathsf{mk}_n[x]^r:=\mathsf{mk}_{[n,n]}[x]^r$ for $n$ in $\mathbb{Z}$.
\end{definition}
  
\begin{remark}\label{rmk::marked_blocks}
  It is easy to see that if $x|_I$ is in $\mathsf{mk}_I[x]^r$ and $J$ is an interval in the right of $I$ then $x|_J$ is in $\mathsf{mk}_J[x]^r$
  and that any block containing a marked block as a suffix is marked.
  Also, any block containing a marked block by $C$ as a prefix is marked by $C$ as well if it appears in $C$. 
\end{remark}

With the notion of marked and transient blocks, define a notion of the length of bridges in nonstop transition.

\begin{lemma}\cite[Lemma 4.18]{AllQ13}\label{lmm::marked_in_point}
  Let $x$ be in $X$.
  There is $n$ in $\mathbb{Z}$ such that $x|_n$ is in $\mathsf{mk}_n[x]^r$.
\end{lemma}

Let $\mathsf{mk}(x)$ denote the smallest $n$ with $x|_n$ in $\mathsf{mk}_n[x]^r$.
Allow it to take $-\infty$ as its value.
  
\begin{remark}\label{rmk::depths_and_degree_and_marked_blokcs}
  The proof of Theorem \ref{thm::depths_and_degree} reveals
  that for any $x_i$ in distinct right classes $C_i$, $1\le i\le d=d_\pi^r(y)$, over $y$ and $M=\max_i\mathsf{mk}(x_i)$ there are infinitely many $N>M$ with $\pi^{-1}(y)|_{[M,N]}$ partitioned into $d$ tangled subsets.
  It is not explicitly stated but implied in the proof of \cite[Theorem 4.22]{AllQ13}, through stages 1 to 3.
\end{remark}

\begin{definition}\label{defn::bridge_length}
  Let $y$ be in $Y$ with nonstop transition $C\to^rC'$ between two right classes $C,C'$over it.
  Given $x$ in $C$, $x'$ in $C'$ and a bridge $\vec{x}$ from $x$ to $x'$ set
  \[ r_x(\vec{x})=\max\{i\in\mathbb{Z}\mid x|_{(-\infty,i]}=\vec{x}|_{(-\infty,i]}\}. \]
  and $r_C(\vec{x})=\max_{x\in C}r_x(\vec{x})$.
  We define the {\em length} $\ell(\vec{x})$ of $\vec{x}$ by $\ell(\vec{x})=\mathsf{mk}(\vec{x})-r_C(\vec{x})$.
  The {\em length} of $C\to^rC'$ is defined to be the supremum of $\ell(\vec{x})$ over all bridges $\vec{x}$ from $C$ to $C'$.
\end{definition}

From the definition it is immediate that for any bridge $\vec{x}$ between distinct classes $\ell(\vec{x})$ is positive and finite.

\begin{lemma}\label{lmm::bridge_length}
  If $y$ is periodic, then the length of any nonstop transition $C\to^rC'$ over $y$ is finite.
  That is, there is a uniform upper bound on $\ell({\vec{x}})$ where $\vec{x}$ varies over all bridges from $C$ to $C'$.
\end{lemma}
\begin{proof}
  Suppose that the claim does not hold.
  For each $n>0$ there are $x_n$ in $C$, $x'_n$ in $C'$ and a bridge $\vec{x}_n$ from $x_n$ to $x'_n$ with $r_{x_n}(\vec{x}_n)=r(\vec{x}_n)$ and increasing $\ell(\vec{x}_n)>n$.
  As $y$ is periodic we may choose periodic $x'_n$, applying Theorem \ref{thm::depths_and_degree}.
  Note that we can take $x'_n$ whose smallest period is less than or equal to thrice the length of a word of depth $d_\pi^r(y)$ appearing in $y$.
  Such periodic points of bounded periods are finitely many,
  so we have $x'$ in $C'$ such that for each $n>0$ there are $x_n$ in $C$ and a bridge $\vec{x}_n$ from $x_n$ to $x'$ with $r_{x_n}(\vec{x}_n)=r(\vec{x}_n)$ and nondecreasing $\ell(\vec{x}_n)>n$.
  
  Refine the selection of $n$ so that all the $r_n=r_{x_n}(\vec{x}_n)$ have the same residue modulo a period of $x'$ for infinitely many $n$.
  Shift all those selected $x_n,x',\vec{x}_n$ and $y$ so that we have $r_n=0$ and $\ell(\vec{x}_n)=\mathsf{mk}(\vec{x}_n)$.

  Set $I_n=[1,\mathsf{mk}(\vec{x}_n)-1]$.
  If there is a bridge from $x'|_{I_n}$ to $\vec{x}_n|_{I_n}$,
  then by Remark \ref{rmk::marked_blocks} $\vec{x}_n|_{\mathsf{mk}(\vec{x}_n)-1}$ is marked by $C'$ at $\mathsf{mk}(\vec{x}_n)-1$.
  It violates the definition of $\mathsf{mk}(\cdot)$.
  On the other hand if there is a bridge from $\vec{x}_n|_{I_n}$ to $x_n|_{I_n}$, then it implies $\vec{x}_n|_1$ appears in $C|_1$.
  This time, a contradiction on $r(\cdot)$ is derived.
  
  Hence there is neither a bridge from $x'|_{I_n}$ to $\vec{x}_n|_{I_n}$ nor a bridge from $\vec{x}_n|_{I_n}$ to $x_n|_{I_n}$ for each selected $n$.
  By compactness, we get limit points $x$ and $\vec{x}$ of those $x_n$'s and $\vec{x}_n$'s, respectively,
  such that $\vec{x}$ is left asymptotic to $x$ but admits no $1$-bridge from $x'$ nor to $x$.
  This $\vec{x}$ is not equivalent to $x$ nor to $x'$,
  however $C\to\vec{C}=[\vec{x}]^r$ and $\vec{C}\to C'$ since $\vec{x}$ is a limit point of $\{\vec{x}_n\}_n$ and $x\to\vec{x}_n$, $\vec{x}_n\to x'$ for each selected $n$.
  Since $C\to^rC'$ is nonstop, $\vec{C}$ clearly may not exist and a contradiction is induced.
\end{proof}

Consider a notion of period for a transition class in a natural way.
A transition class $C$ has {\em period} $p>0$ if $\sigma^p(C)=C$.
When $p=1$, $C$ is called {\em fixed}.
Every transition class over a periodic point is also periodic.

\begin{lemma}\label{lmm::complexity_of_periodic_class}
  Let $y$ be a periodic point in $Y$ and $C$ a right class over $y$ with period $p$.
  Let $C^*$ be a subset of $C$ the language of which consists only of marked blocks in $C$.
  For each fully supported Markov measure $\mu$ on $X$,
  we have $0<\lambda<1$ and $K>0$ with
  \[ \lim_n\frac{\mu[C^*|_{[0,pn-1]}]_0}{\lambda^n}=K. \]
\end{lemma}
\begin{proof}
  Every sufficiently long word which appears in $y$ has depth $d_\pi^r(y)$.
  As $C^*$ is obtained from $C$ by performing a $\sigma$-invariant operation, that is, forbidding transient blocks everywhere,
  we have $\sigma^pC^*=C^*$ as well.
  Remark \ref{rmk::marked_blocks} implies that $C^*$ is nonempty.
  
  Let $\mu$ be an arbitrary fully supported 1-step Markov measure on $X$.
  Let $\mathsf{A}$ the $\mathcal{A}\times\mathcal{A}$ matrix given by $\mathsf{A}_{ab}=\mu[ab]/\mu[a]$ for $a$ and $b$ in $\mathcal{A}$.
  Let $\mathcal{A}^*_n=C^*|_n$ and $\mathsf{A}^*$ be the $\mathcal{A}^*_0\times\mathcal{A}^*_0$ matrix given by
  \[ \mathsf{A}^*_{ab}=\sum_{a_0=a,a_p=b,a_i\in\mathcal{A}^*_i,1\le i\le p-1}\prod_{j=0}^{p-1}A_{a_ja_{j+1}} \]
  for $a$ and $b$ in $\mathcal{A}^*_0$.
  Then $(A^*)^n_{ab}$ is the sum of all the probabilities
   of the paths of length $pn$ from $a$ to $b$ up to $\mu$.
  As $\mu$ is fully supported and $C^*|_{[0,pn]}$ is tangeld for some $n>0$, $\mathsf{A}^*$ is primitive and has a unique Perron eigenvalue $\lambda$.
  So $\lim_n(A^*)^n_{ij}/\lambda^n=v^*_iw^*_j$
  where $v^*$ and $w^*$ are the normalized right and left Perron eigenvectors of $A^*$, respectively.

  As $X$ is mixing, for a left Perron eigenvector $\mathsf{w}$ of $\mathsf{A}$ with $\sum_{a\in\mathcal{A}}\mathsf{w}_a=1$ we have 
  \[ \mu[C^*|_{[0,pn-1]}]_0=\sum_{b\in\mathcal{A}^*_0}\sum_{a\in\mathcal{A}^*_0}\mathsf{w}_a(\mathsf{A}^*)^n_{ab}=\sum_{b\in\mathcal{A}^*_0}(\mathsf{w}_a)_{a\in\mathcal{A}^*_0}(A^*)^n. \]
  Finally, $(\mathsf{w}_a)_{a\in\mathcal{A}^*_0}$ is a linear combination of eigenvectors of $\mathsf{A}^*$,
  thus $\mu[C^*|_{[0,pn-1]}]_0$ grows approximately in the ratio of $\lambda$.
\end{proof}

\begin{lemma}\label{lmm::complexity_of_periodic_class_realized}
  Let $y$ be a periodic point in $Y$ and $C$ a right class over $y$. Let $C^*$ be a subset of $C$ the language of which consists only of marked blocks in $C$.
  Then we have $p$ in $\mathbb{N}$ and $0<\Lambda<1$ such that given $\Lambda\le\lambda<1$ there are a fully supported Markov measure $\mu$ on $X$ and $K>0$ with
  \[ \lim_n\frac{\mu[C^*|_{[0,pn-1]}]_0}{\lambda^n}=K. \]
\end{lemma}
\begin{proof}
  The proof starts with the same idea and setting as in the previous lemma.
  Assume without loss of generlity that $X$ is not trivial.
  Let $\mu$ be an arbitrary fully supported 1-step Markov measure on $X$ and $\mathsf{A}$ the $\mathcal{A}\times\mathcal{A}$ matrix given by $\mathsf{A}_{ab}=\mu[ab]/\mu[a]$ for $a$ and $b$ in $\mathcal{A}$.
  Let $\mathcal{A}^*_n=C^*|_n$ and $\mathsf{A}^*$ be the $\mathcal{A}^*_0\times\mathcal{A}^*_0$ matrix given by
  \[ \mathsf{A}^*_{ab}=\sum_{a_0=a,a_p=b,a_i\in\mathcal{A}^*_i,1\le i\le p-1}\prod_{j=0}^{p-1}A_{a_ja_{j+1}} \]
  for $a$ and $b$ in $\mathcal{A}^*_0$.  
  As in Lemma \ref{lmm::complexity_of_periodic_class}, $\mathsf{A}^*$ is primitive and has a unique Perron eigenvalue $\lambda^*$.
  If $\mu$ is concentrated on the orbit of some periodic point in $C^*$, then $\lambda^*=1$ and $\mu[C^*|_{[0,pn-1]}]_0$ is a constant.
  However, as $X$ is nontrivial, there is a periodic point in $X\setminus C^*$
  and $\lambda^*$ must be smaller than 1 for $\mu$ to be fully supported.
  For a left Perron eigenvector $\mathsf{w}$ of $\mathsf{A}$ with $\sum_{a\in\mathcal{A}}\mathsf{w}_a=1$ we have 
  \[ \mu[C^*|_{[0,pn-1]}]_0=\sum_{b\in\mathcal{A}^*_0}\sum_{a\in\mathcal{A}^*_0}\mathsf{w}_a(\mathsf{A}^*)^p_{ab}=\sum_{b\in\mathcal{A}^*_0}(\mathsf{w}_a)_{a\in\mathcal{A}^*_0}(A^*)^p. \]  
    
  Set $m:=\inf\{\lambda^*\mid\mu\text{ is fully supported}\}<1$.
  Imagine that we change the entries of $\mathsf{A}$ smoothly.
  By the continuity of linear algebra, we can find fully supported 1-step Markov measures on $X$ such that the respected $\lambda^*$ are sufficiently close to $m$ and 1.
  Then given $m<\lambda<1$ the fully supporeted $\mu$ with $\lambda^*=\lambda$ is found by the intermediate value theorem.
  Finally, $(\mathsf{w}_a)_{a\in\mathcal{A}^*_0}$ is a linear combination of eigenvectors of $(\mathsf{A}^*)^p$,
  and $\mu[C^*|_{[0,pn-1]}]_0$ grows approximately in the ratio of $\lambda^n$.
\end{proof}

\begin{remark}\label{rmk::complexity_of_periodic_class_realized}
  All the other right classes on the orbit of $C$ have the same growth rate of complexity as $C$.
  If $C'$ is another right class over $y$ not on the orbit of $C$, that is, $\sigma^jC'\ne C$ for all $j\ge0$,
  then we have $p$ in $\mathbb{N}$ and $0<\Lambda<1$ such that given $\Lambda<\lambda,\lambda'<1$ there are a fully supported Markov measure $\mu$ on $X$ and $K,K'>0$ with
  \[ \lim_n\mu[C^*|_{[0,pn-1]}]_0/\lambda^n=K\text{ and }\lim_n\mu[C'^*|_{[0,pn-1]}]_0/\lambda'^n=K'. \]
 The proof is similar to Lemma \ref{lmm::complexity_of_periodic_class_realized}.
\end{remark}
      
\begin{proposition}\label{prop::transitional_periodic_case}
  If there is a periodic point in $Y$ which admits a transition between distinct rights classes,
  then $X$ admits a fully supported Markov measure which is not sent to a Gibbs measure by $\pi$.
\end{proposition}
\begin{proof}
  Let $|u|=p$ and $y=u^\infty$ denote a periodic point in $Y$ with $y|_{[pn,p(n+1)-1]}=u$ for all $n$ in $\mathbb{Z}$.
  Let $C\to^rC'$ be nonstop transition over $y$.
  Since $y$ is periodic, $C$ and $C'$ are not on the same orbit.
     
  Replacing $u$ with a power of $u$ if needed, we may assume $p>3N$ where
  \[ N:=\max\{\ell(\vec{x})\mid\vec{x}\text{ is a bridge from }C\text{ to }C'\}<\infty. \]
  Let $\vec{C}$ be the set of all bridges $\vec{x}$ from $C$ to $C'$ with $\vec{x}|_I$ transient through $I$ for some $I\subset[1,p]$.
  It is finite.
  We also may assume $\sigma^pD=D$ for any right class $D$ over $y$.
  
  Since there is a point of more than one class $X$ is nontrivial.
  Assume large enough $p$ and apply Remark \ref{rmk::complexity_of_periodic_class_realized} to find a fully supported Markov measure $\mu$ on $X$ and $K,K'>0,0<\varepsilon<\lambda<1$ with
  \[ \lim_n\frac{\mu[C^*|_{[0,pn-1]}]_0}{\lambda^n}=K,\lim_n\frac{\mu[C'^*|_{[0,pn-1]}]_0}{\lambda^n}=K'	 \]
  and
  \[ \lim_n\frac{\mu[D^*|_{[0,pn-1]}]_0}{\varepsilon^n}=O(1)	 \]
  for all the other right classes $D$ over $y$ not on the orbits of $C$ and $C'$.
  
  Every bridge from $C|_{[0,pn-1]}$ to $C'|_{[0,pn-1]}$ has their transient blocks contained in $[Nj,Nj+p-1]$ for some $j$ in $\mathbb{Z}^+$.
  Consider the cylinder sets determined by by all those bridges from $C|_{[0,pn-1]}$ to $C'|_{[0,pn-1]}$.
  The measure of their union is bounded above and below by $\lambda^{n-1}$ multiplied with the number of bridges from $C|_{[0,p-1]}$ to $C'|_{[0,p-1]}$ and with the number of intervals of length $p$ in $[0,pn-1]$, up to constants.
  That is,
  \[ C_1\cdot n\lambda^n<\mu\bigcup_{0\le j\le p(n-1)/N}\mu[\vec{C}|_{[Nj,Nj+p-1]}]_{Nj}<C_2\cdot n\lambda^n \]
  for some $C_1,C_2>0$ and all $n\ge1$.
    
  For $\nu$ to be a Gibbs measure, there must be a continuous function $f:X\to\mathbb{R}$ with
  \begin{equation}\label{eqn::gibbs}
    \frac{\nu[u^n]_1}{\exp S_0^{pn-1}f(y)}=\frac{\nu[u^n]_1}{\tau^n},
  \end{equation}
  where $\tau=\exp S_pf(y)$ is bounded above and below at the same time by positive real numbers.
  However, resetting the values of $C_1,C_2>0$ if needed, we have for all $n\ge1$ 
  \[
    \frac{C_1(2+n)\lambda^n}{\tau^n}<\frac{\nu[u^n]_1}{\tau^n}<p\frac{C_2(2+n)\lambda^n}{\tau^n}.
  \]
  So (\ref{eqn::gibbs}) goes to 0 as $n$ increases if $\tau>\lambda$ and to $\infty$ otherwise because of the summand $n\lambda^n$.
  Therefore $\nu$ fails to be a Gibbs measure.
\end{proof}

Proposition \ref{prop::transitional_periodic_case} is only about transitions between classes over periodic points.
Here we show that bi-continuing factor maps admits such a transition if it does a transition between classes over any points.

\begin{proposition}\label{prop::continuing_transitioncal_case}
  Let $\pi$ be bi-continuing.
  If there is a point in $Y$ which admits a transition between distinct rights classes,
  then $X$ admits a fully supported Markov measure which is not sent to a Gibbs measure by $\pi$.
\end{proposition}
\begin{proof}
  We show that if $\pi$ is bi-continuing and there is a point $y$ in $Y$ with two distinct right classes $C\to^rD$,
  then there is a periodic point $y'$ in $Y$ and two right classes $C'\ne D'$ with $C'\to^rD'$.
  Then Proposition \ref{prop::transitional_periodic_case} completes the proof.
  
  Up to conjugacy, assume that $\pi$ is bi-eresolving.
  Take a point $y$ with $d=d_\pi^r(y)$ distinct right classes $C=C_1\to^rD=C_2,\cdots,C_d$ where $C\to^rD$ is nonstop.
  Choose $x_i$ from each $C_i,1\le i\le d$.
  Select a sequence $\{\vec{x}_n\}_{n\ge1}$ of bridges from $x_1$ to $x_2$ with $r_n=r_{x_1}(\vec{x}_n)<l_n=\ell(\vec{x}_n)<r_{n+1}$.
  By shifting the points, we may assume $r_1=0\ge M=\max_i\mathsf{mk}(x_i)$.
  Further we require that for each $n\ge1$ $\pi^{-1}(y)|_{[M,r_{n+1}]}$ is partitioned into $d$ tangled subsets:
  it is possible by Remark \ref{rmk::depths_and_degree_and_marked_blokcs}.
  Finally, refine the selection of $\vec{x}_n$ so that  $d$-tuples $(x_1|_{r_n},\cdots,x_d|_{r_n}),n\ge1$, are all the same independently of $n$.

  Repeat $x_1|_{[r_1,r_2)},\cdots,x_d|_{[r_1,r_2)}$ and $y|_{[r_1,r_2)}$ to get $x_1',\cdots,x_d'$ and $y'$ which are periodic points given by the following:
  \[ x_i'|_j=x|_r\text{ and }y'|_j=y|_r\text{ if }j\mod r_2\equiv r \]
  for $1\le i\le d,j\in\mathbb{Z}$.
  As we have a bridge from $x'_1|_{[r_1,r_2]}=x_1|_{[r_1,r_2]}$ to $x'_2|_{[r_1,r_2]}=x_2|_{[r_1,r_2]}$,
  there is a transition $x'_1\to^rx'_2$.
  If $x'_2\not\to^rx'_1$, then $y'$ is a desired periodic point with $C'=[x'_1]^r,D'=[x'_2]^r$ and we are done.
  
  Now, suppose to the contrary, that is, $x'_2\to^rx'_1$.
  Find an $r_1$-bridge $\cev{x}$ from $x'_2$ to $x'_1$.
  As $\pi$ is bi-eresolving, we can find a preimage $\bar{x}$ of $y$ with $\bar{x}|_{[r_1,r_2]}=\cev{x}|_{[r_1,r_2]}$.  
  Let $\bar{x}|_{r_2}$ be in $C_k|_{r_2}$ for some $1\le k\le d$.
  If $k=2$ and $\bar{x}|_{r_2}$ stays in $D|_{r_2}$,
  then we reset $\bar{x}$ to be another preimage of $y$ with $\bar{x}|_{[r_1,r_2]}=\cev{x}|_{[lr_2,(l+1)r_2]}$ where $l$ is the smallest natural number with $k\ne2$.
  Such $l$ must exist since $\cev{x}$ is a bridge to $x'_1$.
  
  Note that $\bar{x}|_{[r_1,r_2]}$ is indeed a bridge from $x_2|_{[r_1,r_2]}$ into $C_k|_{[r_1,r_2]}$,
  so $D\to^rC_k$ since $r_1\ge\mathsf{mk}(x_2)$.  
  If $k=1$, then $x_1\sim^rx_2$ and $C=D$, which is a contradiction.
  If $k\ge3$, then $x_2\to^rx_k$, and
  we apply the same argument as above to $x_2$ and $x_k$.
  That is, construct peridoic points $x_2'$ and $x_k'$ repeating $x_2|_{[r_1,r_2]}$ and $x_k|_{[r_1,r_2]}$,
  and see whether $x_k'\to^rx_2'$ or not.
  If not, we found the desired periodic points.
  If so, by the similar argument as in the case of $x_1'$ and $x_2'$
  we get another contradiction $x_2\sim^rx_k$ or another transition from $x_k$ into some transition class over $y$ other than $C,D,C_k$.
  As there are finitely many classes over $y$,
  we may not continue this process forever and in the end will find two non-equivalent periodic points with a transition between them.
\end{proof}

\section{When factor maps reduce periods}
  
In the present section, we show that factor maps does not preserve Gibbsian property if it reduces the periods of right classes.
Throughout the section $\pi$ is always a 1-block factor map from a two-sided 1-step mixing shift of finite type $X$ over $\mathcal{A}$ onto a sofic shift $Y$ over $\mathcal{A}$.

\begin{lemma}\label{lmm::local_retract}
  Let $y$ be a periodic point in $Y$.
  Then we have conjugacies $\phi:X\to\bar{X}$ and $\psi:Y\to\bar{Y}$
  such that $\bar{\pi}:=\psi\circ\pi\circ\phi^{-1}:\bar{X}\to\bar{Y}$ is a 1-block conjugacy with
  \[ \bar\pi^{-1}(\bar{y})|_{[m,n]}=\bar\pi^{-1}(\bar{y}|_{[m,n]}) \]
  for all $\bar{y}$ in $\bar{Y}$ and $m\le n$ in $\mathbb{Z}$.
\end{lemma}
\begin{proof}
  By \cite[Lemma 4.15]{AllQ13}, for each $m,n$ we have $R_{m,n}$ in $\mathbb{Z}^+$ with
  \[ \pi^{-1}(y)|_{[m,n]}=\pi^{-1}(y|_{[m-R_{m,n},n+R_{m,n}]})|_{[m,n]}. \]
  As $y$ is periodic, $\{R_{m,n}\}$ is finite.
  Set $R:=\max_{m,n}R_{m,n}$ so that we have for any $m,n$
  \[ \pi^{-1}(y)|_{[m,n]}=\pi^{-1}(y|_{[m-R,n+R]})|_{[m,n]}. \]
 
  Define $\phi$ on $X$ and $\psi$ on $Y$ by
  \[ \phi(x)|_i=(x|_i,\pi(x)|_{[i-R,i+R]})\text{ and }\psi(y)|_i=y|_{[i-R,i+R]}, \]
  respectively, and set $\bar{X}:=\phi(X)$ and $\bar{Y}:=\psi(Y)$.
  They are clearly injective, and hence, conjugacies onto $\bar{X}$ and $\bar{Y}$, respectively.
  Also $\bar\pi:=\psi\circ\pi\circ\phi^{-1}$ is a 1-block map since given any $\bar{x}=\phi(x)$ in $\bar{X}$ we have $\bar\pi(\bar{x})|_i=\psi(\pi(x))|_i=\pi(x)|_{[i-R,i+R]}$,
  which is just the second component of $\bar{x}|_i$.
  
  Finally, consider any $\bar{y}=\psi(y)$ in $\bar{Y}$.
  Then,
  \[\begin{split}
    \bar\pi^{-1}(\bar{y})|_{[m,n]}&=\phi(\pi^{-1}(y))|_{[m,n]}\\
    &=\{(x|_m,y|_{[m-R,m+R]})\cdots(x|_n,y|_{[n-R,n+R]})\mid x\in\pi^{-1}(y)\}.
  \end{split}\]
  On the other hand,
  \[\begin{split}
    \bar\pi^{-1}(\bar{y}|_{[m,n]})=&\bar\pi^{-1}(y|_{[m-R,n+R]})\\
    =&\{(u|_1,y|_{[m-R,m+R]})\cdots(u|_{[n-m+1]},y|_{[n-R,n+R]})\mid\\
    &\,\,\,\,\,\,\,\,\,\,\,\,\,\,\,\,\,\,\,u\in\pi^{-1}(y|_{[m-R,n+R]})|_{[m,n]}=\pi^{-1}(y)|_{[m,n]}\}\\
    =&\{(x|_m,y|_{[m-R,m+R]})\cdots(x|_n,y|_{[n-R,n+R]})\mid x\in\pi^{-1}(y)\}.
  \end{split}\]
  Therefore $\bar\pi^{-1}(\bar{y})|_{[m,n]}=\bar\pi^{-1}(\bar{y}|_{[m,n]})$.
\end{proof}
  
\begin{lemma}\label{lmm::gibbs_potential_for_periodic_points}
  Let $\nu$ be a Gibbs measure on $Y$ for a normalized potential $f:Y\to\mathbb{R}$
  which is also the image of a Markov measure $\mu$ on $X$ under a factor map $\pi$.
  Let $\pi$ admit no transition between distinct classes.
  Then for a periodic point $y$ of $Y$ with period $p$ we have
  \[ \exp S_0^{p-1}f(y)=\lim_n\frac{\nu[y|_{[0,n]}]_0}{\nu[y|_{[p,n]}]_0}. \]
\end{lemma}
\begin{proof}
  Let $u_n=y|_{[0,n-1]}$ for each $n\ge1$.
  As $y$ is periodic, by Lemma \ref{lmm::local_retract} we may assume up to some conjugacy
  \[ \nu[u_n]_0=\sum_{w\in\pi^{-1}(u_n)}\mu[w]_0=\sum_{C\in\llbracket y\rrbracket^r}\mu[C|_{[0,n-1]}]_0. \]
  As $\pi$ admits no transition between distinct classes, we further have
  \[ \nu[u_n]_0=\sum_C\mu[C^*|_{[0,n-1]}]_0. \]
  By Lemma \ref{lmm::complexity_of_periodic_class}, for each right class $C$ over $y$ we have some $0<\lambda(C)<1$, $K(C)>0$ and $p(C)$ in $\mathbb{N}$ with $\lim_n\mu[C^*|_{[0,p(C)n-1]}]_0/\lambda(C)^n=K(C)$.
  Let $\Lambda=\max_C\lambda(C)^{p/p(C)}$ and $D$ be a right class over $y$ with $\lambda(D)^{p/p(D)}=\Lambda$.
  Let $q=p(D)$.
  All the $\sigma^pD,\sigma^{2p}D,\cdots,\sigma^{q-p}D$'s are distinct but $\lambda({\sigma^pD})=\cdots=\lambda({\sigma^{q-p}D})=\lambda(D)$.
  Then $\mu[D|_{[0,p'n-1]}]_0/\lambda(D)^n$ converges and so do $\mu[\sigma^{jp}D|_{[0,p'n-1]}]_0/\lambda(D)^n$, $0<j<q/p$, as $n$ increases.
  
  Now, let $\lambda=\exp S_0^{p-1}f(y)$.
  There exists $M>0$, by the Gibbs inequality for $\nu$, such that for all $n$ we have
  \[
    \frac1M\lambda^{n-m}<\frac{\nu[u_{pn+l}]_0}{\nu[u_{pm+l}]_0}<M\lambda^{n-m}
  \]
  for all $m<n$ and $0\le l<q$.
  Also we have
  \[ \lim_n\frac{\nu[u_{qn}]_0}{\Lambda^{qn/p}}=K \]
  for some $K>0$.
  Similarly,
  \[ \lim_n\frac{\nu[u_{qn+l}]_0}{\Lambda^{qn/p}}=K_l \]
  for each $0<l<q$ and some $K_l>0$.  
  
  If $m$ and $n$ increase along multiples of $q/p$, then $\nu[u_{pn}]_0/\nu[u_{pm}]_0\to\Lambda^{n-m}$, so $\Lambda=\lambda$.  
  Similarly, $\nu[u_{p(n+1)+l}]_0/\nu[u_{pn+l}]_0\to\lambda$ when $n$ increases along multiples of $q/p$ for each $0<l<q$.
  Hence $\nu[u_{n+p}]_0/\nu[u_n]_0\to\lambda$, that is,
  \[ \lambda=\lim_n\frac{\nu[y|_{[0,n]}]_0}{\nu[y|_{[p,n]}]_p}. \]
\end{proof}

An example is presented to illustrate typical behaviour in period reducing cases.
  
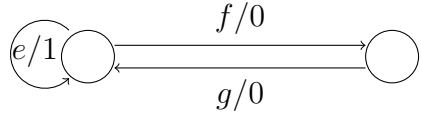
\begin{figure}
  \begin{tikzpicture}[->]
    \node[circle,draw,inner sep=7pt] (I) at (0,2) {};
    \node[circle,draw,inner sep=7pt] (J) at (4,2) {};
      
    \draw (0.35,2.15) to (2,2.15) node[above] {$f/0$} to (3.65,2.15);
    \draw (3.65,1.85) to (2,1.85) node[below] {$g/0$} to (0.35,1.85);
    \draw (-0.25,2.35) arc (45:315:0.45) node[above left] {$e/1$};
  \end{tikzpicture}
  \caption{$0^\infty$ admits a right class of period 2}\label{fig::non-fixed_periodic_case}
\end{figure}

\begin{example}\label{eg::non-fixed_periodic_case}
  In Figure \ref{fig::non-fixed_periodic_case}  $0^\infty$ has two right classes of period 2.
  Define a 1-step fully supported Markov measure $\mu$ on $X$ given by the stochastic matrix
  \[ P=\begin{pmatrix}1-p&p&0\\0&0&1\\1-p&p&0\end{pmatrix}. \]
  Its initial probability vector is given by $((1-p)/(1+p),p/(1+p),p/(1+p))$.
  Let $\nu=\pi\mu$.
  Then
  \[ \frac{\nu[0^{2n+1}]_0}{\nu[0^{2n}]_0}=\frac{p^n\cdot p/(1+p)+p^n\cdot p/(1+p)}{p^{n-1}\cdot p/(1+p)+p^n\cdot p/(1+p)}=\frac{2p}{1+p} \]
  and
  \[ \frac{\nu[0^{2n+2}]_0}{\nu[0^{2n+1}]_0}=\frac{p^n\cdot p/(1+p)+p^{n+1}\cdot p/(1+p)}{p^n\cdot p/(1+p)+p^n\cdot p/(1+p)}=\frac{1+p}2. \]
  By Lemma \ref{lmm::gibbs_potential_for_periodic_points}, for $\nu$ to be a Gibbs measure for some potential $f$,
  $\nu[0^{n+1}]_0/\nu[0^n]_1$ needs to converge.
  If so, then $4p=(1+p)^2$ and $p=1$.
  However, it implies that $\mu$ is not fully supported.
  That is, all the fully supported 1-step Markov measures on $X$ lose their Gibbs property when transformed under $\pi$.
\end{example}

\begin{proposition}\label{prop::period_blowingup_case}
  If there is a periodic point in $Y$ the period of which is strictly smaller than the period of some right class $C$ over it,
  then $X$ admits a fully supported Markov measure which is not sent to a Gibbs measure by $\pi$.
\end{proposition}
\begin{proof}
  By Propositions \ref{prop::transitional_periodic_case} we may assume that no transition is allowed between distinct right classes over periodic points.
  Then, for every right class $C$ over a periodic point of $Y$ we have $C=C^*$.
  Choose a periodic point $y$ in $Y$ with the smallest period $q$ and a right class $C$ of the smallest period $p>q$.
  Say $y|_{[0,p-1]}=w^r$ with $|w|=q$, $r\ge2$ and $p=qr$.
  Apply Lemma \ref{lmm::local_retract} to assume $\pi^{-1}(y)|_{[m,n]}=\pi^{-1}(y|_{[m,n]})$ for all $m,n$ in $\mathbb{Z}$.
  
  We need a little more preparation.
  Taking a higher-block presentation, we may assume that $C|_i$, $0\le i<p$, are all distinct.
  If $Y$ consists of a single fixed point $y$,
  then $X$ has to be its unique transition class which is fixed as well.
  We already excluded this case,
  so $Y$ has more than one point and there is $n$ in $\mathbb{Z}$ such that $C|_n$ has a following symbol which lies outside $C|_{n+1}$ and will be called an {\em escaping} symbol from $C$.
  Shifting $y$ we may assume $n=0$.
  Let $\mu$ a 1-step Markov measure on $X$ which assigns large transition probabilities on escaping symbols from $C$ which follows $C|_0$
  and small transition probabilities on escaping symbols from $C$ which follows $C|_n,$ $0<n<p$:
  Say, the latter probabilities are all 0 and the former ones $0<s<1$.
  That is, once you get in $C$ you stay there until $C|_0=C|_p=\cdots$ is met.
  
  Through Lemma \ref{lmm::complexity_of_periodic_class_realized} and Remark \ref{rmk::complexity_of_periodic_class_realized} we modify $\mu$ a little so that
  \[ \lim_n\mu[C|_{[0,pn-1]}]_0/\lambda^n=K>\lim_n\mu[D|_{[0,pn-1]}]_0/\rho^n\ge0 \]
  for some positive real numbers $\rho<\lambda<1,K$ and any right class $D$ over $y$ not on the orbit of $C$.
  Applying Lemma \ref{lmm::gibbs_potential_for_periodic_points} we see that $\nu[y|_{[0,pn+q-1]}]_0/\nu[y|_{[q,pn+q-1]}]_q$ converges.
  Since $\lambda>\rho>0$,
  \begin{equation}\label{eqn::period_blowingup_ratio_1}
    \lim_n\frac{\nu[y|_{[0,pn+q-1]}]_0}{\nu[y|_{[q,pn+q-1]}]_q}=\lim_n\frac{\nu[y|_{[0,pn+q-1]}]_0}{\nu[y|_{[0,pn-1]}]_0}=\lim_n\frac{\sum_{j=0}^{r-1}\mu[\sigma^{jq}C|_{[0,pn+q-1]}]_0}{\sum_{j=0}^{r-1}\mu[\sigma^{jq}C|_{[0,pn-1]}]_0}.
  \end{equation}
  Similarly,
  \begin{equation}\label{eqn::period_blowingup_ratio_2}
    \lim_n\frac{\nu[y|_{[0,pn+2q-1]}]_0}{\nu[y|_{[q,pn+2q-1]}]_q}=\lim_n\frac{\sum_{j=0}^{r-1}\mu[\sigma^{jq}C|_{[0,pn+2q-1]}]_0}{\sum_{j=0}^{r-1}\mu[\sigma^{jq}C|_{[0,pn+q-1]}]_0}.
  \end{equation}

  Let $P$ and $Q$ be the sums of all initial probabilities of the states in $C|_0$, and of all initial probabilities of the states in $C|_{jq}$, $1\le j\le r-1$, determined up to $\mu$, respectively.
  Then
  \[ \sum_{j=0}^{r-1}\mu[\sigma^{jq}C|_{[0,pn-1]}]_0=(1-s)^n(P+Q). \]
  Also
  \[ \sum_{j=0}^{r-1}\mu[\sigma^{jq}C|_{[0,pn+q-1]}]_0=(1-s)^{n+1}P+(1-s)^nQ \]
  and
  \[
    \sum_{j=0}^{r-1}\mu[\sigma^{jq}C|_{[0,pn+2q-1]}]_0=\begin{cases}
      (1-s)^{n+1}(P+Q)&\text{ if }r=2\\
      (1-s)^{n+1}P+(1-s)^nQ&\text{ otherwise}
    \end{cases}.
  \]
  
  Suppose that $\nu=\pi\mu$ is a Gibbs measure.
  The convergence of $\nu[y|_{[0,m]}]_0/\nu[y|_{[q,m]}]_q$ with respect to $m$ is implied by Lemma \ref{lmm::gibbs_potential_for_periodic_points}.
  So $(\ref{eqn::period_blowingup_ratio_1})=((1-s)P+Q)/(P+Q)$ and
  \[
    (\ref{eqn::period_blowingup_ratio_2})=\begin{cases}
      \frac{(1-s)(P+Q)}{(1-s)P+Q}&\text{ if }r=2\\
      1&\text{ otherwise}
    \end{cases}
  \]
  must coincide.
  If $r\ne2$, then $(\ref{eqn::period_blowingup_ratio_1})=(\ref{eqn::period_blowingup_ratio_2})$ implies that $P=0$ so that $\mu$ cannot be fully supported.
  Otherwise, $r=2$ and we get $(1-s)P^2=Q^2$.
  Even if for some Markov measure on $X$ we have $(1-s)P^2=Q^2$, modifying the value of $s$ a little we obtain a Markov measure $\mu$ on $X$ with $(\ref{eqn::period_blowingup_ratio_1})\ne(\ref{eqn::period_blowingup_ratio_2})$, since $P$ and $Q$ change linearly with respect to $s$ while $(\ref{eqn::period_blowingup_ratio_1})=(\ref{eqn::period_blowingup_ratio_2})$ is a quadratic expression of $P$ and $Q$.
  Now, by the continuity of linear algebraic operations, we are able to construct a fully supported Markov measure $\mu$ on $X$ with $(\ref{eqn::period_blowingup_ratio_1})\ne(\ref{eqn::period_blowingup_ratio_2})$,
  finishing the proof.
\end{proof}

\section{When factor maps are not bi-continuing}
  
In the present section, we show that factor maps does not preserve Gibbsian property if it is not bi-continuing.
Throughout the section $\pi$ is always a 1-block factor map from a two-sided 1-step mixing shift of finite type $X$ over $\mathcal{A}$ onto a sofic shift $Y$ over $\mathcal{A}$.

First, we are going to reduce the non-continuing property to periodic points.
Given a two-sided sequence $y$ we introduce the notions of its initial and eventual class degrees.
Consider transitions and class degrees over one-sided sequences in a natural way:
for left and right infinite sequences, left and right transitions and class degrees, respectively, are well-defined similarly as before.
The {\em initial} and {\em eventual} class degrees of $y$ are defined to be $\lim_{n\to-\infty}d_\pi^l(y|_{(-\infty,n]})$ and $\lim_{n\to\infty}d_\pi^r(y|_{[n,\infty)})$, respectively.
If $y$ has eventual class degree $d$,
then there are preimages $x_1,\cdots,x_d$ of $y$ and $N$ in $\mathbb{Z}$
such that each preimage of $y|_{[n,\infty)}$ is right equivalent to some $x_j|_{[n,\infty)}$, $1\le j\le d$, $d=d_\pi^r(y|_{[n,\infty)})$ and $x_i|_{[n,\infty)}\not\sim^rx_j|_{[n,\infty)}$, $i\ne j$, for any $n\ge N$.
Similar things happen when $y$ has initial class degree $d$.

A point $x$ is called {\em eventually} periodic if for some $n$ in $\mathbb{Z}$ $x|_{[n,\infty)}$ is periodic, and is called {\em initially} periodic if for some $n$ in $\mathbb{Z}$ $x|_{(-\infty,n]}$ is periodic.

For $y$ in $Y$ let $\omega^+(y)$ denote the $\omega$-limit set of $y$, that is, the set of the limit points of $\{\sigma^j\mid j\ge0\}$.
Let $\omega^-(y)$ denote the set of the limit points of $\{\sigma^j(y)\mid j<0\}$.

\begin{lemma}\label{lmm::upper_bound_for_length_of_subword_with_same_degree}
  There is $N>0$ such that any word $u$ of length $n>N$ in $\mathcal{B}(Y)$ has a subword $w$ of length $N$ with $d_\pi(w)=d_\pi(u)$.
\end{lemma}
\begin{proof}
  Suppose not.
  Given any $N$ there is a word $u_n$ in $\mathcal{B}_n(Y)$ for some $n>N$ such that all the subwords of $u_n$ are strictly shallower than $u_n$.
  By compactness, we get a point $y$ in $Y$ with $d_\pi^r(y)>d_\pi(y|_{[m,n]})$ for all $m\le n$ in $\mathbb{N}$.
  It clearly contradicts Theorem \ref{thm::depths_and_degree}.
\end{proof}

\begin{proposition}\label{prop::periodic_continuing}
  The following are equivalent:
  \begin{enumerate}
    \item\label{itm::not_right_continuing} $\pi$ is not right continuing;
    \item\label{itm::not_right_continuing_on_periodic_points} There is $x$ in $X$ such that $\pi(x)$ is left asymptotic to an initially periodic point $y$ in $Y$ but no point left asymptotic to $x$ is sent to $y$;
    \item\label{itm::initial_and_left_class_degrees_disagree} There is $y$ in $Y$ with strictly greater initial class degree than left class degree.
  \end{enumerate} 
\end{proposition}
\begin{proof}
  (\ref{itm::not_right_continuing_on_periodic_points})$\implies$ (\ref{itm::not_right_continuing}): Trivial by definition of continuing property.

  (\ref{itm::not_right_continuing})$\implies$ (\ref{itm::initial_and_left_class_degrees_disagree}):
  There is $x$ in $X$ such that $\pi(x)$ is left asymptotic to a point $y$ in $Y$ but no point left asymptotic to $x$ is sent to $y$.
  Let $N=\max\{n\mid \pi(x)|_{(-\infty,n]}=y|_{(-\infty,n]}\}$.
  Then given any preimage $x'$ of $y$ there is no bridge from $x|_{(-\infty,N]}$ to $x'|_{(-\infty,N]}$
  since such a bridge would make a preimage of $y$ left asymptotic to $x$.
  Hence we have at least one more preimage of $y|_{(-\infty,N]}$ which is not left equivalent to $x'|_{(-\infty,N]}$ for any preimage $x'$ of $y$.
  Immediately, the inital class degree of $y$ is strictly greater than $d_\pi^l(y)$.

  (\ref{itm::initial_and_left_class_degrees_disagree})$\implies$ (\ref{itm::not_right_continuing_on_periodic_points}):
    Let $d=d_\pi^l(y)$.
    Let $d'$ be the initial class degree of $y$ and assume $d'>d$.
    Take preimages $x_1,\cdots,x_d$ of $y$, $x_{d+1},\cdots,x_{d'}$ in $X$ and $N$ in $\mathbb{Z}$
    such that $\pi(x_1)|_{(\infty,-N]}=\cdots=\pi(x_{d'})|_{(\infty,-N]}=y|_{(\infty,-N]}$ and $x_i|_{(\infty,-N]}\not\sim^lx_j|_{(\infty,-N]}$ for any $n\le N$ and $1\le i<j\le d'$.
    
    By an analogue of Theorem \ref{thm::depths_and_degree} for left transitions,
    there are infinitely many $n<N$ and $m<n$ such that $y|_{[m,n]}$ has depth $d'$.
    By Lemma \ref{lmm::upper_bound_for_length_of_subword_with_same_degree} we may bound $m-n$ from above.
    Then there must be a recurrent word in $y|_{(-\infty,N]}$ with depth $d'$.
    Find $m'<n'<m<n<N$ such that $y|_{[m',n']}=y|_{[m,n]}$ has depth $d'$.
    Also we may assume that $y|_{[m',n']}=y|_{[m,n]}$ is fiber-routable through
    \[ M=\{x_1|_{m'+k}=x_1|_{m+k},\cdots,x_{d'}|_{m'+k}=x_{d'}|_{m+k}\} \]
    at some $1<k<n'-m'=n-m$.
    
    Shifting $y$, assume $m'=0$.
    Let
    \[ y'=\begin{cases}
      y|_i&\text{if }i\ge m\\
      y|_{i\mod m}&\text{otherwise}
    \end{cases}
    \text{ and }
    x_j'=\begin{cases}
      x_j|_i&\text{if }i\ge m\\
      x_j|_{i\mod m}&\text{otherwise}
    \end{cases} \]
    for $1\le j\le d'$.
    Clearly $x_j',1\le j\le d'$, and $y'$ are initially periodic and $x_1',\cdots,x_d'$ are preimages of $y'$.
    
    Suppose that a preimage $x'$ of $y'$ is left asymptotic to $x_{d'}'$.
    As $y'|_{[m,n]}=y|_{[m,n]}$ is fiber-routable through $M=\{x_1'|_k,\cdots,x_{d'}'|_k\}$ at $m+k$,
    there is $x_j|_{m+k},1\le j\le d'$ through which $x'|_{[m,n]}$ is routable at $m+k$.
    Assume $x'|_{m+k}=x_s'|_{m+k}=x_s|_{m+k}$.
    If such $s$ is greater than $d$, then $x_s|_{(-\infty,m+k)}x'|_{[m+k,\infty)}$ is a preimage of $y$ left asymptotic to $x_s$,
    which is a contradiction.
    Otherwise, there is $q$ in $\mathbb{Z}$ such that $x'|_{(q-1)m+k}$ meets $x_s'|_{(q-1)m+k}$ for some $s>d$ but $x'|_{qm+k}$ does $x_t'|_{qm+k}$ for some $1\le t\le d$.
    Then $x_s|_{(-\infty,k)}x'|_{[(q-1)m+k,qm+k]}x_t|_{(k,\infty)}$ is a preimage of $y$ left asymptotic to $x_s$ while $s>d$, which is absurd again.    
    
    Hence, for $s>d$ we do not have a preimage of $y'$ left asymptotic to $x_s'$.
\end{proof}

A similar result holds for left continuing factor maps and their eventual class degrees.
  
\begin{lemma}\label{lmm::ratio_of_ergodic_sum}
  Let $\mu$ be a Gibbs measure on a 1-step mixing shift of finite type $X$ for a normalized potential $f$ and Gibbs constants $K_1<K_2$. Let $x$ and $x'$ be in $X$ with $x|_{[0,n-1]}=x'|_{[0,n-1]}$ for some $n\ge0$. Then
  \[ \frac{K_1}{K_2}<\frac{\exp S_0^{n-1}f(x)}{\exp S_0^{n-1}f(x')}<\frac{K_2}{K_1} \]
\end{lemma}
\begin{proof}
  Immediate from
  \[ \frac{\mu[x|_{[0,n-1]}]_0}{K_2}<\exp S_0^{n-1}f(x)<\frac{\mu[x|_{[0,n-1]}]_0}{K_1} \]
  and
  \[ \frac{\mu[x'|_{[0,n-1]}]_0}{K_2}<\exp S_0^{n-1}f(x')<\frac{\mu[x'|_{[0,n-1]}]_0}{K_1}. \]
\end{proof}

An example is presented to illustrate typical behaviour in non-continuing cases.
  
\begin{figure}
  \begin{tikzpicture}[->]
    \node[circle,draw,inner sep=7pt] (I) at (0,2) {};
    \node[circle,draw,inner sep=7pt] (J) at (4,2) {};
    \node[circle,draw,inner sep=7pt] (K) at (0,0) {};
    \node[circle,draw,inner sep=7pt] (L) at (4,0) {};
      
    \draw (0.35,2.15) to (2,2.15) node[above] {$2$} to (3.65,2.15);
    \draw (0.15,0.35) to (0.15,1) node[right] {$1$} to (0.15,1.65);
    \draw (-0.15,1.65) to (-0.15,1) node[left] {$3$} to (-0.15,0.35);
    \draw (3.65,1.85) to (2,1.85) node[below] {$1$} to (0.35,1.85);
    \draw (4.15,0.35) to (4.15,1) node[right] {$2$} to (4.15,1.65);
    \draw (3.85,1.65) to (3.85,1) node[left] {$3$} to (3.85,0.35);
    \draw (4.25,1.65) arc (225:495:0.45) node[below right] {$2$};
    \draw (-0.25,2.35) arc (45:315:0.45) node[above left] {$1$};
    \draw (4.25,-0.35) arc (225:495:0.45) node[below right] {$f/3$};
    \draw (-0.25,0.35) arc (45:315:0.45) node[above left] {$e/3$};
  \end{tikzpicture}
  \caption{Neither left nor right continuing factor map}\label{fig::non-continuing}
\end{figure}
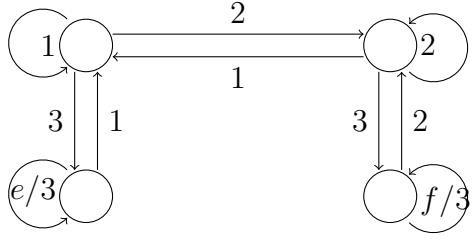

\begin{example}\cite{Kem11}\label{eg::non-continuing}
  The factor map presented in Figure \ref{fig::non-continuing} is neither left nor right continuing:
  $3^\infty$ has two right classes $[e^\infty]^r$ and $[f^\infty]^r$ but $e$ has no following edge labelled 2 while $f$ is followed by an edge labelled 2,
  which shows that $\pi$ is not right continuing.
  In a similar way, $\pi$ is easily shown to be not left continuing.
  
  Define a 1-step fully supported Markov measure $\mu$ on $X$ by putting transition probabilities $p>q$ on $e$ and $f$, respectively, and let $\nu=\pi\mu$.
  Then for some constant $C>0$ $\nu[3^n2]_0/\nu[3^{n+1}]_0$ is smaller than $Cq^n/p^{n+1}$ and goes to 0 as $n$ increases.
  Suppose that $\nu$ is a Gibbs measure for some normalized potential $f$ and Gibbs constants $K_1,K_2>0$. By Lemma \ref{lmm::ratio_of_ergodic_sum} and the inequality (\ref{eqn::gibbs_inequality}), we have
  \[ \frac{\nu[3^n2]_0}{\nu[3^{n+1}]_0 }>\frac{K_1^2}{K_2^2}\exp(\max_xf(x)-\min_xf(x)), \]
  which is absurd.
  Hence, $\nu$ is not a Gibbs measure for any $f$.
\end{example}
  
\begin{proposition}\label{prop::noncontinuing_case}
  Let $\pi$ be not bi-continuing.
  Then $X$ admits a fully supported Markov measure which is not sent to a Gibbs measure by $\pi$.
\end{proposition}
\begin{proof}  
  Assume that $\pi$ is not right continuing.
  By Proposition \ref{prop::periodic_continuing} there are initially periodic points $x_1,x_2$ in $X$
  such that $\pi(x_1)$ is left asymptotic to $y=\pi(x_2)$ but no left asymptotic point of $x_1$ is a preimage of $y$.
  Say $x_1|_{(-\infty,r]}=u^\infty\alpha,x_2|_{(-\infty,r]}=v^\infty\beta$ and $y|_{(-\infty,r]}=w^\infty\delta$
  where $|u|=|v|=|w|=p,|\alpha|=|\beta|=|\delta|=r$ and no follower of $u$ in $\mathcal{B}(X)$ is mapped to $\delta$.

  Let $C=[u^\infty]^r$ and $C'=[v^\infty]^r$.
  By Proposition \ref{prop::transitional_periodic_case} we may assume that $\pi$ does not admit a transition between distinct right classes over a periodic point of $Y$.
  Then $C^*=C$ and $C'^*=C'$.
  If $C$ and $C'$ are on the same orbit, then it implies that $\pi$ reduces the period of $C$ to the period of $y$ so that Proposition \ref{prop::period_blowingup_case} is applied to finish the proof.
  Thus we may assume that $C'$ is not the orbit of $C$.  

  Let $\sigma^jC\ne C'$ for $0< j<p$.
  Let $\mu$ be a fully supported Markov measure on $X$, found by Remark \ref{rmk::complexity_of_periodic_class_realized},
  such that $\lim_n\mu[C|_{[0,pn-1]}]_0/\lambda^n=K,\lim_n\mu[D|_{[0,pn-1]}]_0/\rho^n=K'$ for some positive real numbers $\rho<\lambda<1,K,K'$, and any right class $D$ not on the orbit of $C$ over $y$.
  Let $\nu=\pi\mu$.
  Consider $w^n\delta$ and $w^{n+r}$.
  For all $n\ge1$, $\nu[w^{n+r}]_0$ is larger than or equal to $\mu[C|_{[0,p(n+r)-1]}]_0$.
  On the other hand,
  \[ \nu[w^n\delta]_0<L\cdot\sum_{D\ne\sigma^kC\forall k}\mu[D|_{[0,pn-1]}]_0 \]
  for some constant $L>0$ since $\pi^{-1}[w^n\delta]_0$ meets no orbit of $C$.
  Then for some $M>0$ and large enough $n>0$ we have $\nu[w^n\delta]_0\le M\rho^n$.
  Thus $\nu[w^n\delta]_0/\nu[w^{n+r}]_0$ is smaller than $\rho^n/\lambda^n$ up to a constant and goes to 0 as $n$ increases to $\infty$.
    
  On the while, suppose that $\nu$ is a Gibbs measure for some normalized potential $f:Y\to\mathbb{R}$ and Gibbs constants $K_1<K_2$.
  Then
  \[ \nu[w^n\delta]_0>K_1\exp S_0^{pn-1}f(x)\exp S_0^{pr-1}f(\sigma^{pn}x) \]
  and
  \[ \nu[w^{n+r}]_0<K_2\exp S_0^{pn-1}f(x')\exp S_0^{pr-1}f(\sigma^{pn}x') \]
  where $x|_{[0,p(n+r)-1]}=w^n\delta$ and $x'|_{[0,p(n+r)-1]}=w^{n+r}$.
  By Lemma \ref{lmm::ratio_of_ergodic_sum} we have
  \[ \frac{\nu[w^n\delta]_0}{\nu[w^{n+r}]_0}>\frac{K_1^2}{K_2^2}\exp(pr(\max_xf(x)-\min_xf(x))) \]
  for all $n\ge1$.
  This contradicts the conclusion of the previous paragraph, therefore $\nu$ fails to be a Gibbs measure.
    
  If $\pi$ is not left continuing, then we consider $\delta w^n$ instead of $w^n\delta$.
  A symmetric argument completes the proof in the first case.
\end{proof}
      
\begin{corollary}\label{cor::transitional_case}
  If there is a point in $Y$ which admits a transition between distinct right classes,
  then $X$ admits a fully supported Markov measure which is not sent to a Gibbs measure by $\pi$.
\end{corollary}
\begin{proof}
  Apply Proposition \ref{prop::continuing_transitioncal_case} if $\pi$ is bi-continuing, or Proposition \ref{prop::noncontinuing_case} otherwise.
\end{proof}
  
\begin{corollary}\label{cor::multiple_class_degree}
  If $d_\pi>1$, then $X$ admits a fully supported Markov measure which is not sent to a Gibbs measure by $\pi$. 
\end{corollary}
\begin{proof}
  We may assume by Propositions \ref{prop::noncontinuing_case}, \ref{prop::period_blowingup_case} and Corollary \ref{cor::transitional_case} that $\pi$ is bi-eresolving, that $\pi$ preserves the periods of points
  and that no transition is admitted between right classes.
  We claim that such a factor map has class degree 1.

  Choose any $a$ in $\mathcal{A}$,
  let $\pi^{-1}(a)=\{a_1,\cdots,a_d\}$ and find a shortest path $\alpha_2$ with $a_1\alpha_2a_2$ allowed in $\mathcal{B}(X)$.
  As $\pi$ is bi-eresolving, all the other $a_j,j\ne1$, have respective predecessors and successors with the same image as $u=\pi(\alpha_2)$
  and, conversely, every preimage of $u$ has some preimages of $a$ as its preceding and succeeding symbols.
  So for each $a_j,1\le j\le d$ there is a cycle $\beta_j$ which is a preimage of a power of $v_2=ua$ and ends with $a_j$.

  Let $y=v_2^\infty$ and $x_j=\beta_j^\infty$, $1\le j\le d$.
  Consider $\beta_1^\infty\alpha_2a_2\beta_2^\infty$.
  Either it is a bridge from $x_1$ to $x_2$, so $x_1\to^rx_2$,
  or $\alpha_2a_2$ is a prefix of a power of $\beta_1$, a suffix of a power of $\beta_2$ and $x_2$ is a nontrivial shift of $x_1$.
  The latter case is in fact not allowed since such $\pi$ would reduce the periods of $[x_1]^r$.
  Only the former case is valid.
  Since no transition exists between right classes,
  $x_1\sim^rx_2$ and for some $N_2>0$ $\beta_1^{N_2}$ and $\beta_2^{N_2}$ are routable through a single symbol at some index.
  Immediately, all the preimages of $v_2^{N_2}$ following $a_1$ and $a_2$ are routable through a single symbol.
  
  For every $j=3,\cdots,d$, find $\alpha_j$ with $a_1\alpha_ja_j$ in $\mathcal{B}(X)$ and let $v_j=\pi(\alpha_j)a$.
  In a similar way as above, for some $N_j>0$ all the preimages of $v_j^{N_j}$ following $a_1$ and $a_j$ are routable through a single symbol at some index.
 
  Set $N:=\max_{2\le j\le d}N_j$ and consider $v=v_2^Nv_3^N\cdots v_d^N$.
  For any preimage $u$ of $v$ with $u|_{N|v_2|}=a_2$ there is a bridge from $u|_{[1,v_2^N]}$ to some power of $\beta_1$,
  since if $u$ follows $a_j$ then there is a bridge to $\beta_1$ from $\beta_j$.
  In turn, for any preimage $w$ of $v$ with $w|_{N|v_2v_3|}=a_3$, there is a bridge from $w|_{[1,v_2^Nv_3^N]}$ to some preimage of $v_2^Nv_3^N$ ending with $a_1$.
  In a similar way, from all the preimage of $v$ following $a_j$ $2\le j\le d$, there are bridges to some preimages of $v$ ending with $a_1$ and vice versa.
  Thus $\tau_\pi(v)=1<d_\pi$.
\end{proof}

\section{Remarks}

We have seen that a factor map from a mixing shift of finite type can preserve Gibbisian property only if it is bi-continuing, preserves the smallest periods of periodic right classes and admits no transition between distinct right classes.

\begin{definition}\label{defn::nearly_fiber-mixing}
  A factor map $\pi$ is said to be {\em nearly fiber-mixing} if it is bi-continuing, preserves the smallest periods of periodic right classes and admits no transition between distinct right classes.
\end{definition}

By Corollary \ref{cor::multiple_class_degree} a nearly fiber-mixing factor map always has class degree 1.
Fiber-mixing factor maps are nearly fiber-mixing and
finite-to-one nearly fiber-mixing factor maps are conjugacies.
In \cite{Pir18} it was shown that a fiber-mixing factor map sends a Gibbs measure with a H\"older continuous potential on a mixing shift of finite type to a Gibbs measure on its image shift.
Here we show that there exists actually a nearly fiber-mixing factor map which is not fiber-mixing.
  
\begin{figure}
  \begin{tikzpicture}[->]
    \node[circle,draw,inner sep=7pt] (I) at (0,0) {};
    \node[circle,draw,inner sep=7pt] (I') at (0,3) {};
    \node[circle,draw,inner sep=7pt] (J) at (3,1.5) {};
    \node[circle,draw,inner sep=7pt] (K) at (8,1.5) {};
      
    \draw (0.35,2.75) to (1.5,2.2) node[below] {$1$} to (2.65,1.65);
    \draw (0.35,0.25) to (1.5,0.8) node[above] {$1$} to (2.65,1.35);
    \draw (3.4,1.5) to (5.5,1.5) node[above] {$2$} to (7.6,1.5);
    \draw (7.6,1.7) to (4,2.45) node[above] {$0$} to (0.4,3.2);
    \draw (7.6,1.3) to (4,0.55) node[below] {$0$} to (0.4,-0.2);
    \draw (-0.25,0.35) arc (45:315:0.45) node[above left] {$a$};
    \draw (-0.25,0.45) arc (45:315:0.65) node[below left] {$b$};
    \draw (-0.25,3.35) arc (45:315:0.45) node[above left] {$a$};
    \draw (-0.25,3.45) arc (45:315:0.65) node[below left] {$b$};
  \end{tikzpicture}
  \caption{Nearly fiber-mixing factor map}\label{fig::nearly_fiber-mixing_factor_map}
\end{figure}
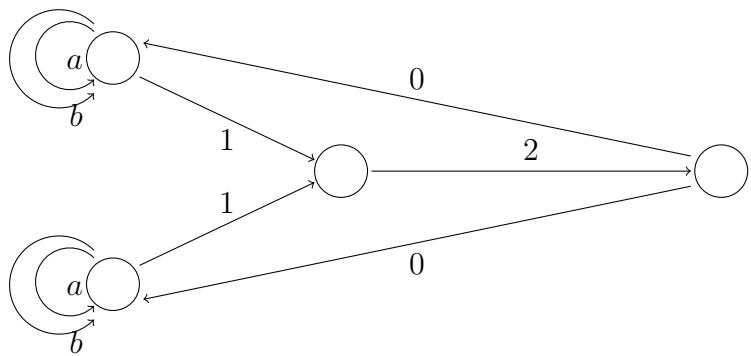
  
\begin{example}\label{eg::nearly_fiber-mixing_factor_map}
  Let $X$ be the underlying edge shift given in  Figure \ref{fig::nearly_fiber-mixing_factor_map}, $\pi$ the labelling map and $Y$ the sofic shift given by the labelled graph.
  Then $\pi$ is nearly fiber-mixing.
\end{example}

\bibliographystyle{amsplain}
\bibliography{ref}

\providecommand{\bysame}{\leavevmode\hbox to3em{\hrulefill}\thinspace}
\providecommand{\MR}{\relax\ifhmode\unskip\space\fi MR }
\providecommand{\MRhref}[2]{%
  \href{http://www.ams.org/mathscinet-getitem?mr=#1}{#2}
}
\providecommand{\href}[2]{#2}
\begin{thebibliography}{10}

\bibitem{AllHJ15}
Mahsa Allahbakhshi, Soonjo Hong, and Uijin Jung, \emph{Class-closing factor
  codes and constant-class-to-one factor codes from shifts of finite type},
  Dyn. Syst. \textbf{30} (2015), no.~4, 485--500. \MR{3430312}

\bibitem{AllHJ15E}
MAHSA ALLAHBAKHSHI, SOONJO HONG, and UIJIN JUNG, \emph{Structure of transition
  classes for factor codes on shifts of finite type}, Ergodic Theory and
  Dynamical Systems \textbf{35} (2015), 2353--2370.

\bibitem{AllQ13}
Mahsa Allahbakhshi and Anthony Quas, \emph{Class degree and relative maximal
  entropy}, Trans. Amer. Math. Soc. \textbf{365} (2013), no.~3, 1347--1368.
  \MR{3003267}

\bibitem{Bow08}
Rufus Bowen, \emph{Equilibrium states and the ergodic theory of {A}nosov
  diffeomorphisms}, revised ed., Lecture Notes in Mathematics, vol. 470,
  Springer-Verlag, Berlin, 2008, With a preface by David Ruelle, Edited by
  Jean-Ren{\'e} Chazottes. \MR{2423393 (2009d:37038)}

\bibitem{BoyT84}
Mike Boyle and Selim Tuncel, \emph{Infinite-to-one codes and markov measures},
  Trans. Amer. Math. Soc. \textbf{285} (1984), no.~2, 657--684. \MR{752497
  (86b:28024)}

\bibitem{ChaU03}
J.~R. Chazottes and E.~Ugalde, \emph{{Projection of Markov Measures May Be
  Gibbsian}}, Journal of Statistical Physics \textbf{111} (2003), no.~5-6,
  1245--1272.

\bibitem{Jun11}
Uijin Jung, \emph{On the existence of open and bi-continuing codes}, Trans.
  Amer. Math. Soc. \textbf{363} (2011), no.~3, 1399--1417. \MR{2737270
  (2012b:37035)}

\bibitem{Kem11}
T.~M~W Kempton, \emph{{Factors of Gibbs measures for subshifts of finite
  type}}, Bulletin of the London Mathematical Society \textbf{43} (2011),
  no.~4, 751--764.

\bibitem{LinM95}
Douglas Lind and Brian Marcus, \emph{An introduction to symbolic dynamics and
  coding}, Cambridge University Press, Cambridge, 1995. \MR{1369092
  (97a:58050)}

\bibitem{JozCK98}
J{\'o}zsef L{\H{o}}rinczi, Christian Maes, and Koen Vande~Velde,
  \emph{Transformations of {G}ibbs measures}, Probab. Theory Related Fields
  \textbf{112} (1998), no.~1, 121--147. \MR{1646444 (99i:60174)}

\bibitem{Pir18}
Mark Piraino, \emph{Projections of gibbs states for h{\"o}lder potentials},
  Journal of Statistical Physics \textbf{170} (2018), no.~5, 952--961.

\bibitem{Sin72}
Ja.~G. Sina{\u\i}, \emph{Gibbs measures in ergodic theory}, Uspehi Mat. Nauk
  \textbf{27} (1972), no.~4(166), 21--64. \MR{0399421 (53 \#3265)}

\bibitem{vanFS93}
Aernout C.~D. van Enter, Roberto Fern{\'a}ndez, and Alan~D. Sokal,
  \emph{Regularity properties and pathologies of position-space
  renormalization-group transformations: scope and limitations of {G}ibbsian
  theory}, J. Statist. Phys. \textbf{72} (1993), no.~5-6, 879--1167.
  \MR{1241537 (94m:82012)}

\bibitem{Ver11}
Evgeny Verbitskiy, \emph{{On factors of g-measures}}, Indagationes Mathematicae
  \textbf{22} (2011), no.~3-4, 315--329.

\bibitem{Yoo10}
Jisang Yoo, \emph{On factor maps that send {M}arkov measures to {G}ibbs
  measures}, J. Stat. Phys. \textbf{141} (2010), no.~6, 1055--1070. \MR{2740403
  (2011k:37005)}

\bibitem{Yoo15}
\bysame, \emph{On the retracts and recodings of continuing codes}, Bull. Korean
  Math. Soc. \textbf{52} (2015), no.~4, 1375--1382.

\end{thebibliography}

\end{document}